\documentclass{amsart}

\usepackage{cite}
\usepackage[utf8]{inputenc}
\usepackage[T1]{fontenc} 
\usepackage[english]{babel} 
\usepackage{amsmath,amsfonts,amsthm, mathrsfs}
\usepackage{ytableau}
\usepackage{graphics}
\usepackage{hyperref}

\allowdisplaybreaks

\theoremstyle{plain}
\newtheorem{theorem}{Theorem}[section]

\newtheorem{lemma}{Lemma}[section]
\newtheorem{proposition}{Proposition}[section]

\newtheorem{remark}{Remark}[section]
 
\theoremstyle{definition}

\newtheorem{example}{Example}[section]
\numberwithin{equation}{section}

\DeclareMathOperator{\GL}{GL}
\DeclareMathOperator{\Sp}{Sp}
\DeclareMathOperator{\SL}{SL}
\DeclareMathOperator{\Tr}{Tr}

\title{	
Random walk on the symplectic forms over a finite field
}

\author{Jimmy He}
\thanks{This research was supported in part by NSERC}
\thanks{The author thanks Persi Diaconis for introducing the problem to him and for helpful discussions}
\address{Department of Mathematics, Stanford University, Stanford, CA  94305}
\email{jimmyhe@stanford.edu}

\begin{document}
\begin{abstract}
Random transvections generate a walk on the space of symplectic forms on $\mathbf{F}_q^{2n}$. The main result is to establish cutoff for this Markov chain. After $n+c$ steps, the walk is close to uniform while before $n-c$ steps, it is far from uniform. The upper bound is proved by explicitly finding and bounding the eigenvalues of the random walk. The lower bound is found by showing that the support of the walk is exponentially small if only $n-c$ steps are taken. The result can be viewed as a $q$-deformation of a result of Diaconis and Holmes on a random walk on matchings.
\end{abstract}
\maketitle
\section{Introduction}
In this paper a random walk on the Gelfand pair $\GL_{2n}(\mathbf{F}_q)/\Sp_{2n}(\mathbf{F}_q)$, which may be identified with the space of symplectic forms on $\mathbf{F}_q^{2n}$, is analyzed. This walk is a $q$-deformation of a walk on random matchings studied by Diaconis and Holmes \cite{DH02}. The eigenvalues of the random walk are obtained through a connection with the random transvection walk studied by Hildebrand \cite{H92} and cutoff is obtained.

A good overview of the use of Gelfand pairs to analyze Markov chains can be found in \cite{D88} or \cite{CST08}. The theory of Gelfand pairs was previously used to study the Bernoulli-Laplace model \cite{DS87}. There are also analogues in the continuous setting, see for example \cite{M14} where cutoff for Brownian motion on Riemannian symmetric spaces is proven. While this work analyzes just one family of Markov chains on the space of symplectic forms, comparison techniques developed in \cite{DSC93} allow upper bounds to be obtained for other walks on the same space.

A transvection on a vector space $V$ is a linear map of the form $I+vf$ for $v\in V$ and $f\in V^*$. Fix some symplectic form $\omega$ on $\mathbf{F}_q^{2n}$ and let $K\subseteq \GL_{2n}(\mathbf{F}_q)$ denote the subgroup preserving $\omega$. Then the transvections that do not preserve $\omega$ form a single double coset, denoted $Kg_\mu K$.

Let $U$ denote the uniform measure on $\GL_{2n}(\mathbf{F}_q)$, $P$ denote the uniform measure on $Kg_\mu K$, and $D$ denote the uniform measure on matrices of the form $\mathrm{diag}(\alpha,1,...,1)$, for $\alpha\in\mathbf{F}_q^*$. Use $\|\cdot \|$ to denote the total variation norm for measures. The main results are the following theorems.

\begin{theorem}
\label{thm:upper bound}
There exists constants $A,B>0$ such that the upper bound
\begin{equation*}
\|P^{\ast (n+c)}\ast D-U\|\leq Ae^{-Bc}
\end{equation*}
holds for all $c\geq 0$ and for all sufficiently large $n$.
\end{theorem}

\begin{theorem}
\label{thm:lower bound}
There exists constants $A,B>0$ such that the lower bound
\begin{equation*}
\|P^{\ast (n-c)}\ast D-U\|\geq 1-Ae^{-Bc}
\end{equation*}
holds for all $c\geq 0$ and for all sufficiently large $n$.
\end{theorem}
Together, these theorems establish cutoff for the Markov chain at $n$. These results are a $q$-analogue of a result for the random walk on matchings analyzed in \cite{DH02}. The relationship between the results in \cite{H92} and the results of this paper mirror the relationship between the random transposition walk in \cite{DS81} and the random walk on matchings in \cite{DH02}, which can be viewed as the Gelfand pair $S_{2n}/B_{n}$ ($B_{n}$ denotes the hyperoctahedral group).

Specifically, the random transposition walk of \cite{DS81}, generated by the uniform distribution on transpositions in $S_n$, is related through the representation theory of $S_n$ with Schur functions. The random walk on matchings studied in \cite{DH02} is related through the representation theory of the Gelfand pair $S_{2n}/B_n$ with Jack polynomials of parameter $2$, and the eigenvalues of the two walks are related. The mixing times of the two walks are both at $1/2n\log n$.

The random transvection walk studied in \cite{H92} is naturally seen as a $q$-deformation of the random transposition walk on $S_n$, as evidenced by the similarity of the formulas for the eigenvalues of the random walk. However, the representation theory of $\GL_n(\mathbf{F}_q)$ is still connected with Schur functions. The random walk studied in this paper is related through the representation theory of the Gelfand pair $\GL_{2n}(\mathbf{F}_q)/\Sp_{2n}(\mathbf{F}_q)$ with Macdonald polynomials of parameter $(q,q^2)$, as explained in \cite{H19}. The Macdonald polynomials are a two parameter deformation of Schur functions and as $q\to 1$, the Macdonald polynomials of parameter $(q,q^\alpha)$ become the Jack polynomials of parameter $\alpha$. The eigenvalues of the walk are also related to the eigenvalues of the random transvection walk in \cite{H92}, mirroring what happens in \cite{DH02} (compare Proposition \ref{prop: transvec to non-symp transvec} with Proposition 1 of \cite{DH02}). Finally, the mixing time for the walk occurs at $n$, the same as for the random transvection walk, again mirroring what occurs for matchings.

The random walk on matchings in \cite{DH02} had various manifestations, including a random walk on phylogenetic trees and on partitions. It is hoped that the walk analyzed in this paper can find similar applications.

The paper is organized as follows. In Section \ref{section: preliminaries}, the notation used in the paper and the necessary background material on Markov chains and Gelfand pairs is reviewed. In Section \ref{sec:computation of spherical function values}, the eigenvalues of the random walk are computed. In Sections \ref{sec: upper bound} and \ref{sec: lower bound}, the upper and lower bounds are established.

\section{Preliminaries}
\label{section: preliminaries}

\subsection{Representation theory of \texorpdfstring{$\GL_n(\mathbf{F}_q)$}{GLn(Fq)}}
To fix notation, the representation theory of $\GL_n(\mathbf{F}_q)$ is briefly reviewed. The representation theory of $\GL_n(\mathbf{F}_q)$ was developed by Green in \cite{G55} but this section follows Macdonald \cite{M79} and his conventions are used. Let $M$ denote the group of units of $\overline{\mathbf{F}}_q$ and let $M_n$ denote the fixed points of $F^n$, where $F$ the Frobenius endomorphism $F(x)=x^q$. Let $L$ be the character group of the inverse limit of the $M_n$ with norm maps between them. Note that $M_n$ can be identified with $\mathbf{F}_{q^n}^*$. The Frobenius endomorphism $F$ acts on $L$ in a natural manner, and there is a natural pairing of $L_n$ with $M_n$ for each $n$.

The $F$-orbits of $M$ can be viewed as irreducible polynomials over $\mathbf{F}_q$ under $O\mapsto \prod_{\alpha\in O}(x-\alpha)$. Denote by $O(M)$ and $O(L)$ the $F$-orbits in $M$ and $L$ respectively. Use $\mathcal{P}$ to denote the set of partitions. Then the conjugacy classes of $\GL_n(\mathbf{F}_q)$ are indexed by partition-valued functions $\mu:O(M)\rightarrow \mathcal{P}$ such that 
\begin{align*}
\|\mu\|=\sum _{f\in O(M)}d(f)|\mu(f)|=n,
\end{align*}
where $d(f)$ denotes the degree of $f$. This is because $\mu$ contains the information necessary to construct the Jordan canonical form. That is, given $\mu$, construct a matrix in $\GL_{n}(\overline{\mathbf{F}_q})$ in Jordan form by taking for each orbit $f\in O(M)$, $l(\mu(f))$ blocks, of sizes $\mu(f)_i$, for each root of $f$. The resulting matrix has $d(f)$ blocks of size $\mu(f)_i$ for each $f$ and $i$, and adding this all up gives $\|\mu\|=n$.

For example, the partition-valued function corresponding to the set of transvections, which have Jordan form
\begin{align*}
\left(\begin{array}{ccccc}
1&1&0&\dots&0\\
0&1&0&\dots&0\\
0&0&1&\dots&0\\
\vdots&\vdots&\vdots&\ddots&\vdots\\
0&0&0&\dots&1
\end{array}\right),
\end{align*}
correspond to the partition-valued function $\mu$ with $\mu(f_1)=(21^{n-2})$ and $\mu(f)=0$ for $f\neq f_1$ (here $f_1$ denotes the minimal polynomial of $1$). Use $q_f$ to denote $q^{d(f)}$. There is a formula for the sizes of conjugacy classes given by
\begin{equation*}
|C_\mu|=\frac{|\GL_n(\mathbf{F}_q)|}{a_\mu(q)},
\end{equation*}
where
\begin{equation*}
a_\mu(q)=q^n\prod _{f\in O(M)}q_f^{2n(\mu(f))}\prod_{i\geq 1}\prod _{j=1}^{m_i(\mu(f))}(1-q_f^{-j}),
\end{equation*}
with $n(\lambda)=\sum (i-1)\lambda_i$ and $m_i(\lambda)$ denoting the number of $i$'s occurring in $\lambda$.

Similarly, the irreducible characters of $\GL_n(\mathbf{F}_q)$ are indexed by functions $\lambda:O(L)\rightarrow \mathcal{P}$ such that
\begin{equation*}
\|\lambda\|=\sum _{\varphi\in O(L)}d(\varphi)|\lambda(\varphi)|,
\end{equation*}
where $d(\varphi)$ denotes the size of the orbit $\alpha$. The dimension of the irreducible representation corresponding to $\lambda$ is given by
\begin{equation*}
d_{\lambda}=\psi_n(q)\prod _{\varphi\in O(L)}q_\varphi^{n(\lambda(\varphi)')}H_{\lambda(\varphi)}(q_\varphi)^{-1},
\end{equation*}
where $\psi_n(q)=\prod _{i=1}^n(q^i-1)$, $q_\varphi=q^{d(\varphi)}$ and $H_\lambda(t)=\prod _{x\in \lambda}(t^{h(x)}-1)$, $h(x)$ denoting the hook length. Note that with this convention, the trivial representation corresponds to the partition-valued function $\lambda(\chi_{1})=(1^n)$ ($\chi_1$ being the trivial character) and $0$ otherwise.

\subsection{Gelfand pairs}
A (finite) \emph{Gelfand pair} is a finite group $G$, with a subgroup $K\subseteq G$ such that inducing the trivial representation from $K$ to $G$ gives a multiplicity-free representation (or equivalently by Frobenius reciprocity, the restriction of any irreducible representation from $G$ to $K$ has at most a $1$-dimensional $K$-fixed subspace). This property means the harmonic analysis on $G/K$ is simplified, allowing the random walk to be analyzed. See \cite{D88} or \cite{CST08} for basic facts about Gelfand pairs and their application to Markov chains.

For a Gelfand pair $G/K$, any representation $\rho$ of $G$ has either no non-zero $K$-fixed vectors, or a $1$-dimensional subspace fixed pointwise by $K$. Say that $\rho$ is a \emph{spherical representation} if it has a $K$-fixed vector, and define the corresponding \emph{spherical function} to be $\phi(g)=\langle v_\rho,\rho(g)v_\rho\rangle$, where $v_\rho$ is a unit $K$-fixed vector. The spherical functions can also be computed by averaging characters over $K$. That is, $\phi(g)=|K|^{-1}\sum _{k\in K}\chi(kg)$ for $\chi$ the character of $\rho$ (if $\rho$ is not a spherical representation, then this average is $0$).

Let
\begin{equation*}
J=\left(\begin{array}{cc}
0&I\\
-I&0\\
\end{array}\right)
\end{equation*}
and define the standard symplectic form on $\mathbf{F}_q^{2n}$ to be $\omega(x,y)=x^TJy$. Then $\Sp_{2n}(\mathbf{F}_q)$ denotes the subgroup of $\GL_{2n}(\mathbf{F}_q)$ which preserves $\omega$, or in other words, for which $g^TJg=J$. To simplify the notation, write $\GL_n$ for $\GL_n(\mathbf{F}_q)$ and $\Sp_{2n}$ or $K$ for $\Sp_{2n}(\mathbf{F}_q)$.

The relevant Gelfand pair for this paper is $\GL_{2n}(\mathbf{F}_q)/\Sp_{2n}(\mathbf{F}_q)$, whose double cosets can be identified with the set of conjugacy classes in $\GL_n(\mathbf{F}_q)$ \cite{BKS90}.

Denote by $\phi_\lambda$ the spherical function corresponding to the partition-valued function $\lambda$ (see \cite{BKS90}, although note a different convention is used in this paper so all partitions labeling representations are transposed). For a partition $\lambda$, let $\lambda\cup\lambda$ denote the partition which contains every part of $\lambda$ twice. The double cosets of $\Sp_{2n}(\mathbf{F}_q)$ are indexed by $\mu:O(M)\rightarrow\mathcal{P}$, with $\|\mu\|=n$, with the matrices
\begin{equation*}
g_\mu=\left(\begin{array}{cc}
M_\mu &0\\
0& I\\
\end{array}\right)
\end{equation*}
being double coset representatives, where $M_\mu$ an $n\times n$ matrix in the conjugacy class of $\GL_n(\mathbf{F}_q)$ corresponding to $\mu$. 

Bannai, Kawanaka and Song worked out the representation theory of $\GL_{2n}(\mathbf{F}_q)/\Sp_{2n}(\mathbf{F}_q)$, including a formula for the spherical functions. However, the formula given is an alternating one unsuitable for asymptotic analysis. The formula obtained in this paper can be proven using the results in \cite{BKS90} with some work, see \cite{H19}, but a probabilistic proof is preferred to minimize the technical machinery needed.

The key result from \cite{BKS90} that is needed is reproduced below. It relates the sizes of double cosets in $\GL_{2n}(\mathbf{F}_q)/\Sp_{2n}(\mathbf{F}_q)$ to the sizes of conjugacy classes in $\GL_n(\mathbf{F}_q)$.

If $f(q)$ is a rational function in $q$, define $f(q)_{q\mapsto q^2}=f(q^2)$. For example,
\begin{equation*}
|\GL_n(\mathbf{F}_q)|_{q\mapsto q^2}=\prod_{i=0}^{n-1}(q^{2n}-q^{2i}).
\end{equation*}

\begin{proposition}[Bannai, Kawanaka, and Song, Proposition 2.3.6\cite{BKS90}]
\label{prop: double coset size}
Let $\mu:O(M)\rightarrow \mathcal{P}$ with $\|\mu\|=n$. Then
\begin{equation*}
|Kg_\mu K|=|K||C_\mu|_{q\mapsto q^2},
\end{equation*}
where $Kg_\mu K$ denotes the double coset indexed by $\mu$ in $\GL_{2n}(\mathbf{F}_q)$ and $C_\mu$ denotes the conjugacy class indexed by $\mu$ in $\GL_n(\mathbf{F}_q)$.
\end{proposition}

\subsection{Random walk on groups}
For any finite group $G$, a random walk can be created from a probability measure on $G$ given by multiplying random elements from this distribution. That is, if $P$ is a probability measure on $G$, then the random walk generated by $P$, starting from the identity, is the sequence of random variables $g_0, g_0g_1, g_0g_1g_2, \dotsc$, where $g_0=e$ and the $g_i$ are independent copies with distribution $P$.

The most well-studied class of random walks on groups are the ones generated by conjugacy-invariant measures, with the prototypical example being the random transposition walk on the symmetric group. An example that will be important in the eigenvalue computations is the random transvection walk studied by Hildebrand \cite{H92}.

The \emph{random transvection walk} is a walk on $\GL_{n}(\mathbf{F}_q)$ generated by the uniform measure on the conjugacy class of transvections. In the course of establishing cutoff for this random walk, Hildebrand found a formula for the eigenvalues of its transition matrix through the character theory of $\GL_{n}(\mathbf{F}_q)$.

Another class of interesting random walks on groups are bi-invariant walks. If $G$ is a finite group and $K$ a subgroup, then $P$ is a \emph{bi-invariant} measure on $G$ if $P(x)=P(kxk')$ for all $x\in G$ and $k,k'\in K$. This walk naturally descends to the quotient $G/K$ in two distinct ways. Namely, there is the more obvious left walk given by the sequence $g_0 K,g_1g_0K,\dotsc$ in $G/K$, but there is also a right walk given by $g_0K,g_0g_1K,\dotsc$ despite the fact that there is no right action by $G$. In many settings, the right walk is the more natural walk. 

It is important to note that although the two walks are distinct in general, if $P$ is invariant under inversion then at the group level the left and right walks are identical. The marginal distributions (of the current state) of both the left and right walks are identical even though the transition probabilities are very different, and so in particular the mixing times are identical.

Let $V$ be a vector space over $\mathbf{F}_q$ and let $\omega$ be a symplectic form on $V$ (that is, a non-degenerate alternating form). Note that any transvection in $\GL(V)$ can be written in the form $I+vf$ for $v\in V$ and $f\in V^*$ with $f(v)=0$. Since $\omega$ is a non-degenerate form, a transvection can always be written as $I+v\omega(w,\cdot)$. A \emph{symplectic transvection} is a transvection in $\Sp_{2n}$. These are of the form $I+\alpha v\omega(v,\cdot)$ for $\alpha\in \mathbf{F}_q$ and $v\in V$ (see \cite[1.4.12]{O78}).

The random walk generated by non-symplectic transvections is the random walk generated by the uniform measure on the set of non-symplectic transvections. This measure is bi-invariant under $\Sp_{2n}(\mathbf{F}_q)$. The random walk is most naturally desccribed using the right walk on the space $\GL_{2n}(\mathbf{F}_q)/\Sp_{2n}(\mathbf{F}_q)$, which can be viewed as the space of symplectic forms on $\mathbf{F}_q^{2n}$. The random walk moves from one symplectic form $\omega$ to another by picking a transvection $X$ uniformly from those not fixing $\omega$ (i.e. non-symplectic ones). The left walk can be described as picking a non-symplectic transvection with respect to a fixed symplectic form $\omega_0$, and acting on $\omega$. Thus, the right walk always moves at every step while the left walk can remain stationary.

Some initial randomness is needed as otherwise only matrices with determinant $1$ would be reached so the walk is started from $g_0$, the diagonal matrix with all $1$ except for one entry, which is uniformly chosen from $\mathbf{F}_q^*$. Of course, a walk on $\SL_{2n}(\mathbf{F}_q)/\Sp_{2n}(\mathbf{F}_q)$ with the same mixing properties could be analyzed instead, but because the representation theory of $\GL_{2n}(\mathbf{F}_q)$ is crucial, the walk on $\GL_{2n}(\mathbf{F}_q)$ is analyzed at the cost of some initial randomness. 

\begin{example}[Transition matrices for $\GL_4(\mathbf{F}_2)/\Sp_{4}(\mathbf{F}_2)$]
\label{ex: trans matrix}
This example gives the explicit transition matrices for $\GL_4(\mathbf{F}_2)/\Sp_{4}(\mathbf{F}_2)$. Note that $n=2$ is the first non-trivial case because $\Sp_{2}(\mathbf{F}_q)\cong \SL_2(\mathbf{F}_q)$ and so the walk is trivial. The computations are done on the double-coset space to minimize the number of states, but they can be done on the coset space or even the group as well.

Note that the double cosets are labeled by partition-valued functions such that $\|\mu\|=2$. This implies that $\mu$ is non-zero only at degree $1$ and $2$ polynomials. There is one degree one irreducible polynomial, $x+1$, and one irreducible degree two polynomial, $x^2+x+1$. A partition-valued function may send $x+1$ to a partition with two boxes, so there are two choices, and must send $x^2+x+1$ to one box, so there is one choice. Thus, there are $3$ double cosets. The following elements $g_i$,
\begin{align*}
g_1=\left(\begin{array}{cccc}
1&0&0&0\\
0&1&0&0\\
0&0&1&0\\
0&0&0&1\\
\end{array}\right),
\\g_2=\left(\begin{array}{cccc}
1&1&0&0\\
0&1&0&0\\
0&0&1&0\\
0&0&0&1\\
\end{array}\right),
\\g_3=\left(\begin{array}{cccc}
0&1&0&0\\
1&1&0&0\\
0&0&1&0\\
0&0&0&1\\
\end{array}\right),
\end{align*}
give representatives for the double cosets coming from the representatives for the conjugacy classes in $\GL_{2n}(\mathbf{F}_2)$. 

First note that since the field is $\mathbf{F}_2$, there is no need to randomize the determinant. The walk is generated by $g_2$, and so if the current position is at $g_1$, it will always move to $g_2$. For the other states, the GAP computer algebra system was used to compute the transition probabilities. To compute the probability of moving from $Kg_iK$ to $Kg_jK$, compute the proportion of elements in $Kg_2Kg_iK$ which lie in $Kg_jK$. This gives the following transition matrix:
\begin{equation*}
S=\left(\begin{array}{ccc}
0&1&0\\
\frac{1}{15}&\frac{6}{15}&\frac{8}{15}\\
0&\frac{2}{3}&\frac{1}{3}\\
\end{array}\right).
\end{equation*}
\end{example}

\subsection{Cutoff for Markov chains}
Given two measures $P,Q$ on a probability space $\Omega$, let the \emph{total variation distance} be defined by
\begin{equation*}
\|P-Q\|=\sup _{A}|P(A)-Q(A)|.
\end{equation*}

The cutoff phenomenon is a particularly sharp convergence of a Markov chain to its stationary distribution within a certain window. More precisely, define the mixing time of a Markov chain $P$ with stationary distribution $U$ by
\begin{equation*}
t_{mix}(\varepsilon)=\max_{x\in \Omega}\min \{t\in\mathbf{N}|\|P^t(x,\cdot)-U\|\leq \varepsilon\},
\end{equation*}
where $P^t(x,\cdot)$ denotes the measure given by taking $t$ steps in the Markov chain, starting at $x$. A family of Markov chains $P_n$ is said to have \emph{cutoff} if $t^{(n)}_{mix}(\varepsilon)/t^{(n)}_{mix}(1-\varepsilon)\rightarrow 1$ as $n\rightarrow \infty$ for all $\varepsilon\in (0,1)$. Note that in the setting of interest on a group, left-invariance implies that the mixing time is independent of the starting state. For a more comprehensive overview on cutoff for Markov chains, see \cite{LP17}.

\section{Computation of eigenvalues}
\label{sec:computation of spherical function values}
In this section, the eigenvalues of the transition matrix for the random walk which are needed to compute the upper bound on the total variation distance are computed. First, a connection is made between the walk on symplectic forms and the random transvection walk in \cite{H92}. This gives a way to compute the eigenvalues directly from the character values found in \cite{H92}, analogous to how the eigenvalues for a random walk on matchings can be found using the ones for random transposition \cite{DH02}. First, the claim made when defining the random walk that the non-symplectic transvections form a single double coset is shown.

\begin{lemma}
The non-symplectic transvections form a single double coset $Kg_\mu K$.
\end{lemma}
\begin{proof}
Notice that any transvection which is not symplectic is of the form $I+v\omega(w,\cdot)$ for $v,w$ linearly independent and $\omega(w,v)=0$. Then for any map $L$ such that $Le_1=v$ and $Le_{n+2}=w$ and which is symplectic,
\begin{equation*}
\begin{split}
L^{-1}(I+v\omega(w,\cdot))L&=I+L^{-1}vw^TJL
\\&=I+L^{-1}vw^T (L^{-1})^TJ
\\&=I+e_1e_{n+2}^TJ
\\&=g_\mu.
\end{split}
\end{equation*}
But such $L$ can always be found because $\omega(v,w)=0$, and so $v,w$ can be extended to a symplectic basis. Thus, any non-symplectic transvection lies in $Kg_\mu K$.
\end{proof}

\begin{proposition}
\label{prop: transvec to non-symp transvec}
Let $T$ denote the transition matrix for the walk on $\GL_{2n}(\mathbf{F}_q)/\Sp_{2n}(\mathbf{F}_q)$ induced by the random transvection walk on $\GL_{2n}(\mathbf{F}_q)$ and $S$ denote the transition matrix for the random non-symplectic transvection walk on $\GL_{2n}(\mathbf{F}_q)/\Sp_{2n}(\mathbf{F}_q)$. Then $T=aS+bI$, where $a$ is the proportion of transvections which are non-symplectic and $b$ denotes the proportion which are symplectic, and
\begin{equation*}
a=\frac{q(q^{2n-2}-1)}{q^{2n-1}-1}
\end{equation*}
and
\begin{equation*}
b=\frac{q-1}{q^{2n-1}-1}.
\end{equation*}
\end{proposition}

\begin{proof}
f a random transvection is picked uniformly, one in $\Sp_{2n}(\mathbf{F}_q)$ is picked with probability $b$, and if not, then it must lie in $Kg_\mu K$ and this happens with probability $a$. Thus, $T=aS+bI$.

The number of transvections in $\GL_{2n}(\mathbf{F}_q)$ is 
\begin{equation*}
\begin{split}
\frac{|\GL_{2n}(\mathbf{F}_q)|}{a_\mu(q)}&=\frac{q^{n(2n-1)}\prod (q^i-1)}{q^{2n+(2n-1)(2n-2)}(1-q^{-1})\prod_{i=1}^{2n-2} (1-q^{-i})}
\\&=\frac{(q^{2n}-1)(q^{2n-1}-1)}{q-1}.
\end{split}
\end{equation*}
The number of symplectic transvections (transvections in $\Sp_{2n}(\mathbf{F}_q)$) is $(q^{2n}-1)$. To see this, write it as $I+\alpha v\omega(v,\cdot)$ and note that $(\alpha,v)$ and $(\alpha/\beta^2,\beta v)$ give the same map (note if $v$ and $w$ are linearly independent, then $(\alpha,v)$ and $(\beta,w)$ must give different maps because they have different kernels). This gives the stated proportions.
\end{proof}

For a partition $\lambda$, and a box $s\in \lambda$, define the \emph{arm length} $a_\lambda(s)$ to be the number of boxes to the right of the box $s$ in $\lambda$, and similarly define the \emph{leg length}  $l_\lambda(s)$ to be the number of boxes below $s$ (when the partition is clear, $\lambda$ will be omitted). Now let
\begin{equation*}
\begin{split}
c_\lambda(q,t)&=\prod _{s\in \lambda}(1-q^{a(s)}t^{l(s)+1}),
\\c'_\lambda(q,t)&=\prod _{s\in \lambda}(1-q^{a(s)+1}t^{l(s)}).
\end{split}
\end{equation*}
Also, for partitions $\mu\subseteq \lambda$, define
\begin{equation*}
\psi'_{\lambda/\mu}=\prod _{s\in C_{\lambda/\mu}\setminus R_{\lambda/\mu}}\frac{b_\lambda(s;q,t)}{b_\mu(s;q,t)},
\end{equation*}
where
\begin{equation*}
b_\lambda(s;q,t)=\frac{1-q^{a(s)}t^{l(s)+1}}{1-q^{a(s)+1}t^{l(s)}},
\end{equation*}
and where $C_{\lambda/\mu}$ denotes the columns of $\lambda$ intersecting $\lambda/\mu$ and similarly $R_{\lambda/\mu}$ but for rows. This notation comes from the theory of Macdonald polynomials, although nothing more than these definitions will be needed.

The following formula for the eigenvalues of the random walk may now be proven. This also gives a combinatorial formula for the spherical functions evaluated on a non-symplectic transvection.
\begin{proposition}
\label{prop: spherical function value}
The eigenvalues of $S$, the transition matrix for the random walk on symplectic forms, are indexed by $\lambda:O(L)\to \mathcal{P}$ such that $\|\lambda\|=n$, with
\begin{equation}
    \label{eq:spherical function value}
    \phi_\lambda=\frac{q^{2n-2}(q^2-1)}{(q^{2n}-1)(q^{2n-2}-1)}\left(\sum _{\lambda_1}\frac{c'_{\lambda}(q,q^2)\psi_{\lambda/\lambda_1}'}{c'_{\lambda_1}(q,q^2)(1-q)}q^{n(\lambda_1')-n(\lambda')}-\frac{q^{2n}-1}{q^{2n-2}(q^2-1)}\right),
\end{equation}
where the sum is over $\lambda_1\subseteq \lambda$ obtained by removing a single box from some $\lambda(\varphi)$ with $d(\varphi)=1$. The eigenvalue $\phi_lambda$ has multiplicity $d_{\lambda\cup\lambda}$.
\end{proposition}
\begin{proof}
First, note that the group $\GL_{2n}(\mathbf{F}_q)$ acts transitively on the space of symplectic forms, and there is a multiplicity-free decomposition
\begin{equation*}
    \mathbf{C}(\GL_{2n}(\mathbf{F}_q)/\Sp_{2n}(\mathbf{F}_q))\cong \bigoplus_{\|\lambda\|=n}V_{\lambda\cup \lambda},
\end{equation*}
where $V_{\lambda\cup\lambda}$ is the irreducible representation of $\GL_{2n}(\mathbf{F}_q)$ indexed by $\lambda\cup\lambda$ \cite[Theorem 4.1.1]{BKS90}. This implies that the random walk given by random transvection on the space of symplectic forms has eigenvalues $\chi_{\lambda\cup\lambda}/d_{\lambda\cup\lambda}$ with multiplicity $d_{\lambda\cup\lambda}$. Then by Proposition \ref{prop: transvec to non-symp transvec}, the eigenvalues of $S$ are given by
\begin{equation*}
    \frac{1}{a}\frac{\chi_{\lambda\cup\lambda}(g_\mu)}{d_{\lambda\cup\lambda}}-\frac{b}{a}
\end{equation*}
with multiplicity $d_{\lambda\cup\lambda}$. 

The corresponding character ratio value at a transvection for $\GL_n(\mathbf{F}_q)$, which can be found in \cite{H92}, is given by
\begin{equation*}
\frac{\chi_{\lambda}(g_\mu)}{d_\lambda}=\frac{q^{2n-1}(q-1)}{(q^{2n}-1)(q^{2n-1}-1)}\left(\sum _{\lambda_1}\frac{\delta(S_{\lambda_1})}{(q-1)\delta(S_\lambda)}-\frac{q^{2n}-1}{q^{2n-1}(q-1)}\right),
\end{equation*}
where $\delta(S_\lambda)=\prod _{\varphi\in O(L)}q_\varphi^{n(\lambda(\varphi)')}H_{\lambda(\varphi)}(q_\varphi)^{-1}$. Then compute
\begin{equation*}
\begin{split}
&\frac{1}{a}\frac{\chi_{\lambda\cup\lambda}(g_\mu)}{d_{\lambda\cup\lambda}}-\frac{b}{a}
\\=&\frac{q^{2n-2}(q^2-1)}{(q^{2n}-1)(q^{2n-2}-1)}\left(\sum _{\lambda_1}\frac{\delta(S_{\lambda_1})}{(q^2-1)\delta(S_{\lambda\cup\lambda})}-\frac{q^{2n}-1}{q^{2n-2}(q^2-1)}\right).
\end{split}
\end{equation*}

Now there is a bijection between boxes which can be removed from $\lambda$ and boxes which can be removed from $\lambda\cup\lambda$ (the first copy of any part can never have a box removed), and so the sum can be rewritten from being over $\lambda_1$ having one box removed from $\lambda\cup\lambda$, to summing over $\lambda_1$ having one box removed from $\lambda$ (see Figure \ref{fig:correspondence between removable boxes} for an example).

\begin{figure}
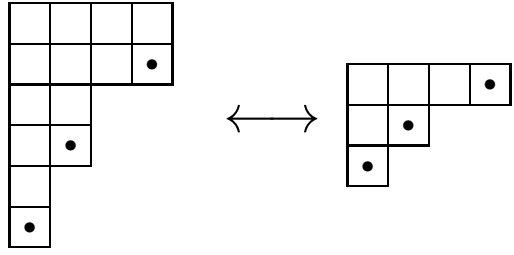

\centering
\begin{tabular}{ccc}
\parbox[c]{2.5cm}{\ydiagram{4,4,2,2,1,1}*[\bullet]{0,3+1,0,1+1,0,1}}&\scalebox{2}{$\longleftrightarrow$}&\parbox[c]{2.5cm}{\ydiagram{4,2,1}*[\bullet]{3+1,1+1,1}}
\end{tabular}
\caption{Correspondence between removable boxes in $\lambda\cup\lambda$ and $\lambda$}
\label{fig:correspondence between removable boxes}
\end{figure}

Then note that the hook length polynomial can be broken up into two factors, coming from even and odd parts, and this corresponds to the two cases $a_\lambda(s)+2l_\lambda(s)+1$ and $a_\lambda(s)+2l_\lambda(s)+2$ (the $2$ in front of the $l_\lambda$ corresponds to doubling the number of boxes in a column, the $1$ or $2$ at the end corresponds to the even and odd boxes).

That is, fix a removable box $s\in \lambda(\varphi)$, which corresponds to a removable box $s\in \lambda\cup\lambda(\varphi)$, and then write (noting that in the ratio, the only factors which matter lie in either the column or row of the removed box)
\begin{equation*}
\begin{split}
&\frac{\delta(S_{\lambda_1})}{(q^2-1)\delta(S_{\lambda\cup\lambda})}
\\=&\frac{q^{n(\lambda_1(\varphi)')}}{q^{n((\lambda(\varphi)\cup\lambda(\varphi))')}(1-q)(q^2-1)}
\\&\qquad\times\prod _{s\in C\setminus R}\frac{1-q^{a_{\lambda\cup\lambda(\varphi)}(s)+l_{\lambda\cup\lambda}(s)+1}}{1-q^{a_{\lambda\cup\lambda}(s)+l_{\lambda\cup\lambda}(s)}}\prod _{s\in R\setminus C}\frac{1-q^{a_{\lambda\cup\lambda}(s)+l_{\lambda\cup\lambda}(s)+1}}{1-q^{a_{\lambda\cup\lambda}(s)+l_{\lambda\cup\lambda}(s)}}.
\end{split}
\end{equation*}
Here $C$ and $R$ denote the column and row in $\lambda\cup\lambda$. By the discussion above there is a correspondence between removed boxes in $\lambda$ and $\lambda\cup\lambda$. Then compute all expressions involving $\lambda\cup\lambda$ in terms of $\lambda$, writing $\lambda_1=\lambda\cup \lambda_2$, giving
\begin{equation*}
\begin{split}
&\frac{q^{n(\lambda_1')}}{q^{n((\lambda\cup\lambda)')}(1-q)}\prod _{s\in C\setminus R}\frac{(1-q^{a_\lambda(s)+2l_\lambda(s)+2})(1-q^{a_\lambda(s)+2l_\lambda(s)+1})}{(1-q^{a_\lambda(s)+2l_\lambda(s)+1})(1-q^{a_\lambda(s)+2l_\lambda(s)})}
\\&\qquad\qquad\times\prod _{s\in R\setminus C}\frac{1-q^{a_\lambda(s)+la_\lambda(s)+1}}{1-q^{a_\lambda(s)+2l_\lambda(s)}}
\\=&\frac{q^{n(\lambda_2')}}{q^{n(\lambda')}(1-q)}\prod _{s\in C\setminus R}\frac{(1-q^{a_\lambda(s)+2l_\lambda(s)+2})}{(1-q^{a_\lambda(s)+2l_\lambda(s)})}\prod _{s\in R\setminus C}\frac{1-q^{a_\lambda(s)+2l_\lambda(s)+1}}{1-q^{a_\lambda(s)+2l_\lambda(s)}}
\end{split}
\end{equation*}
with $q^{n(\lambda_1')-n((\lambda\cup\lambda)')}=q^{n(\lambda_2')-n(\lambda')}$ because $n(\lambda)$ can be computed by placing $i-1$ in the $i$th row of $\lambda$ and summing all these values and so it's clear that the extra boxes do not matter. Then
\begin{equation}
\begin{split}
    \label{eq: first term}
    &\frac{\delta(S_{\lambda_1})}{(q^2-1)\delta(S_{\lambda\cup\lambda})}
    \\=&\frac{q^{n(\lambda_2')}}{q^{n(\lambda')}(1-q)}\prod _{s\in C\setminus R}\frac{(1-q^{a_\lambda(s)+2l_\lambda(s)+2})}{(1-q^{a_\lambda(s)+2l_\lambda(s)})}\prod _{s\in R\setminus C}\frac{1-q^{a_\lambda(s)+2l_\lambda(s)+1}}{1-q^{a_\lambda(s)+2l_\lambda(s)}}.
\end{split}
\end{equation}

Now compute
\begin{equation}
\label{eq:columns-and-rows}
\begin{split}
&\frac{c'_{\lambda}(q,q^2)}{c'_{\lambda_2}(q,q^2)(1-q)}\psi'_{\lambda/\lambda_2}
\\=&(q-1)^{-1}\prod_{s\in C\cup R} \frac{1-q^{a_{\lambda}(s)+2l_{\lambda}(s)+1}}{1-q^{a_{\lambda_2}(s)+2l_{\lambda_2}(s)+1}}\prod _{s\in C\setminus R}\frac{1-q^{a_{\lambda}(s)+2l_{\lambda}(s)+2}}{1-q^{a_{\lambda_2}(s)+2l_{\lambda_2}(s)+2}}\frac{1-q^{a_{\lambda_2}(s)+2l_{\lambda_2}(s)+1}}{1-q^{a_{\lambda}(s)+2l_{\lambda}(s)+1}}
\\=&\prod _{C\setminus R}\frac{1-q^{a_{\lambda}(s)+2l_{\lambda}(s)+2}}{1-q^{a_{\lambda_2}(s)+2l_{\lambda_2}(s)+2}}\prod _{s\in R\setminus C}\frac{1-q^{a_{\lambda}(s)+2l_{\lambda}(s)+1}}{1-q^{a_{\lambda_2}(s)+2l_{\lambda_2}(s)+1}},
\end{split}
\end{equation}
where $C$ and $R$ denote the column and row that the removed box is contained in.

Note here $a_\lambda(s)$ denotes the arm length of $s$ in $\lambda(\varphi)$, and similarly for $l_\lambda(s)$ (strictly speaking, $\lambda$ is not a partition but a function into the set of partitions, but because only one box is removed, only one partition is relevant to the ratio).

Finally, from \eqref{eq:columns-and-rows} and \eqref{eq: first term} the result follows.
\end{proof}

\begin{remark}
The combinatorial formula given by Proposition \ref{prop: transvec to non-symp transvec} was also obtained in \cite{H19} using the theory of Macdonald polynomials and some formulas for spherical function values due to Bannai, Kawanaka and Song \cite{BKS90}. Here, a combinatorial proof is preferred to reduce the necessary background and highlight the similarities with the Weyl group case studied in \cite{DH02} where a similar proof was found.
\end{remark}

\begin{example}[Eigenvalues for $\GL_4(\mathbf{F}_2)/\Sp_{4}(\mathbf{F}_2)$]
To make Proposition \ref{prop: spherical function value} more concrete, return to the setting of Example \ref{ex: trans matrix}. The eigenvalues for the transition matrix computed there are $1,-\frac{1}{3},\frac{1}{15}$, which match the values given by Equation \eqref{eq:spherical function value} as shown below.

There are three partition-valued functions $O(L)\rightarrow \mathcal{P}$ with $\|\lambda\|=2$, two supported at an orbit of degree $1$ and one at an orbit of degree $2$. If the eigenvalues are labeled $\phi_1,\phi_2,\phi_3$ (with $\phi_1$ being the trivial one), then
\begin{align*}
\phi_1&=1,
\\ \phi_2&=\frac{1}{15},
\\ \phi_3&=-\frac{1}{3}.
\end{align*}
\end{example}

\section{Upper Bounds}
\label{sec: upper bound}
The goal of this section is to establish Theorem \ref{thm:upper bound}. By the upper bound lemma (see \cite{DS87}, and also \cite{D88} or \cite{CST08}),
\begin{equation*}
\|P^{\ast k}\ast D-U\|^2\leq \frac{1}{4}\sum _{\rho}d_\rho \Tr\left(\widehat{D}(\rho)\widehat{D}(\rho)^*(\widehat{P}(\rho)^k)^*\widehat{P}(\rho)^k\right),
\end{equation*}
where for a probability measure $Q$, $\widehat{Q}(\rho)=\sum Q(g)\rho(g)$ is the Fourier transform (which turns convolution into a product). Let $h_\alpha$ denote the matrix $\mathrm{diag}(\alpha,1,\dotsc,1)$.  Then
\begin{equation*}
\begin{split}
&\Tr\left(\widehat{D}(\rho_\lambda)\widehat{D}(\rho_\lambda)^*(\widehat{P}(\rho_\lambda)^k)^*\widehat{P}(\rho_\lambda)^k\right)
\\=&|\phi(g_\mu)|^{2k}\langle v_{\rho_\lambda},\widehat{D}(\rho_\lambda)\widehat{D}(\rho_\lambda)^*v_{\rho_\lambda}\rangle
\\=&|\phi(g_\mu)|^{2k}(q-1)^{-2}\sum _{\alpha,\beta\in \mathbf{F}_q^*}\langle v_{\rho_\lambda},\rho_\lambda(h_\alpha h_\beta^{-1})v_{\rho_\lambda}\rangle
\\=&(q-1)^{-1}|\phi_\lambda(g_\mu)|^{2k}\sum _{\alpha\in \mathbf{F}_q^*}\phi_\lambda(h_\alpha)
\end{split}
\end{equation*}
by picking a basis $v_1,...,v_n$ for the representation with $v_1=v_{\rho_\lambda}$ a unit vector fixed by $\Sp_{2n}$ (if no such vector exists, then $\widehat{P}(\rho_\lambda)=0$), such that $\widehat{P}(\rho_\lambda)^k=\mathrm{diag}(\phi_\lambda(g_\mu)^k,0,...,0)$ (see for example, \cite[Proposition 4.7.2]{CST08}). Here the $\phi_\lambda$ are the spherical functions associated to the representations indexed by $\lambda$, and can be computed as $\phi_\lambda(g)=\langle v{\rho_\lambda},\rho(g)v_{\rho_\lambda}\rangle$. The inner product is chosen so that $\rho$ is a unitary representation.

Now $|\phi_\lambda(h_\alpha)|\leq 1$ (spherical functions are always pointwise bounded as $|\langle v_{\rho_\lambda},\rho(g)v_{\rho_\lambda}\rangle|\leq 1$ since $\rho(g)$ is unitary), and so
\begin{equation*}
\Tr\left(\widehat{D}(\rho_\lambda)\widehat{D}(\rho_\lambda)^*(\widehat{P}(\rho_\lambda)^k)^*\widehat{P}(\rho_\lambda)^k\right)\leq |\phi_\lambda(g_\mu)|^{2k}
\end{equation*}
giving the bound
\begin{equation}
\label{eq: DS bound}
\|P^{\ast k}\ast D-U\|^2\leq \frac{1}{4}\sum _{\lambda}d_{\lambda\cup\lambda} |\phi_{\lambda}|^{2k},
\end{equation}
where from now on $\phi_\lambda$ will be used to denote the eigenvalue $\phi_\lambda(g_\mu)$.

From \cite{H92}, $\Tr\left(\widehat{D}(\rho_\lambda)^*\widehat{D}(\rho_\lambda)\right)=0$ (and thus $\widehat{D}(\rho_\lambda)=0$) if $\lambda$ is a partition-valued function taking the value $(1^{2n})$ for some $\alpha$ of degree $1$ so these terms may be ignored. Actually these terms should be thought of as measuring the randomness of the determinant and explain why the initial randomness is needed.

The strategy and estimates to bound this sum are similar to those in \cite{H92}, with the appropriate modifications. The sum is split into three parts and each one is bounded individually. First, bound the spherical functions attaining negative values, then bound the the ones with small length, and finally bound the rest.

First, begin with some preliminary estimates which will be useful later. The following is a trivial bound for the negative terms.

\begin{lemma}
For all $\lambda:O(L)\rightarrow \mathcal{P}$, $\phi_\lambda\geq -\frac{1}{q^{2n-2}-1}$.
\end{lemma}
\begin{proof}
The first term in equation \ref{eq:spherical function value} is positive (even though the $c'_\lambda$ are potentially negative, there are an equal number of factors so the signs cancel) and the second term is $\frac{-1}{q^{2n-2}-1}$.
\end{proof}

\begin{lemma}
Let $\lambda:O(L)\rightarrow \mathcal{P}$. Then
\begin{equation}
\label{eq:1}
\phi_{\lambda}\leq \frac{q^{2n-2}(q^2-1)}{(q^{2n}-1)(q^{2n-2}-1)}\sum _{d(\varphi)=1}\sum _{(i,j)\in C(\lambda(\varphi))}\frac{(q^{2i}-1)(q^{j}-1)}{q^{j-1}(q-1)(q^2-1)},
\end{equation}
where $C(\lambda)$ denotes the set of removable boxes of $\lambda$.
\end{lemma}

\begin{proof}
First, suppose that $\lambda_1\subseteq \lambda$ is obtained by removing the box in row $i$ and column $j$. Then the inequality
\begin{equation*}
\frac{c'_{\lambda}(q,q^2)}{c'_{\lambda_1}(q,q^2)(1-q)}\psi'_{\lambda/\lambda_1}\leq \frac{(q^{2i}-1)(q^{j}-1)}{(q-1)(q^2-1)}.
\end{equation*}
holds. To see this, first note that for $z\leq x\leq y$,
\begin{equation*}
\frac{q^{y+z}-1}{q^y-1}\leq \frac{q^{x+z}-1}{q^x-1}
\end{equation*}
and so
\begin{equation*}
\begin{split}
&\frac{c'_{\lambda}(q,q^2)}{c'_{\lambda_1}(q,q^2)(1-q)}\psi'_{\lambda/\lambda_1}
\\=&\prod _{C\setminus R}\frac{1-q^{a_{\lambda}(s)+2l_{\lambda}(s)+2}}{1-q^{a_{\lambda_1}(s)+2l_{\lambda_1}(s)+2}}\prod _{s\in R\setminus C}\frac{1-q^{a_{\lambda}(s)+2l_{\lambda}(s)+1}}{1-q^{a_{\lambda_1}(s)+2l_{\lambda_1}(s)+1}}
\\\leq &\frac{(q^4-1)...(q^{2i}-1)}{(q^2-1)...(q^{2(i-1)}-1)}\frac{(q^2-1)...(q^{j}-1)}{(q-1)...(q^{j-1}-1)},
\end{split}
\end{equation*}
where Equation \eqref{eq:columns-and-rows} is used.

Finally, since
\begin{equation*}
q^{n(\lambda_1(\varphi))-n(\lambda(\varphi))}=1/q^{j-1}
\end{equation*}
if $(i,j)$ is removed from $\lambda(\varphi)$ to obtain $\lambda_1(\varphi)$, the result follows.
\end{proof}

\begin{lemma}
\label{lemma: negative case}
For the negative eigenvalues, the bound
\begin{equation*}
\sum _{\phi_{\lambda}\leq 0}d_{\lambda\cup\lambda}|\phi_{\lambda}|^{2(n+c)}\leq Ae^{-Bc}
\end{equation*}
holds for constants $A,B>0$, and for sufficiently large $n$.
\end{lemma}
\begin{proof}
Note that $|\phi_{\lambda}|\leq (q^{2n-2}-1)^{-1}$ and 
\begin{equation*}
\begin{split}
\sum _{\phi_{\lambda}\leq 0}d_{\lambda\cup\lambda}&\leq \sum _{\phi_{\lambda}\leq 0}d_{\lambda\cup\lambda}^2
\\&\leq |\GL_{2n}(\mathbf{F}_q)/\Sp_{2n}(\mathbf{F}_q)|
\\&\leq q^{2n^2}
\end{split}
\end{equation*}
and so
\begin{equation*}
\sum _{\phi_\lambda\leq 0}d_{\lambda\cup\lambda}|\phi_\lambda|^{2(n+c)}\leq q^{2n^2}(q^{2n-2}-1)^{-2(n+c)}.
\end{equation*}
The exponential bound is an easy consequence.
\end{proof}

Next, consider the terms with small length. To establish the bound, the following lemmas are useful.

\begin{lemma}
Let $\lambda$ be a partition valued function with $l(\lambda(\varphi))=n-i$ for some $d(\varphi)=1$. Let $\widetilde{\lambda}$ be the partition valued function equal to $\lambda$ except at $\varphi$, where the first column is removed. Then
\begin{equation*}
d_{\lambda\cup\lambda}\leq \frac{q^{2i}\prod _{j=2i+1}^{2n} (q^j-1)}{\prod _{j=1}^{2n-2i}(q^j-1)}d_{\widetilde{\lambda}\cup\widetilde{\lambda}}.
\end{equation*}
\end{lemma}
\begin{proof}
This follows from \cite[Lemma 5.3]{H92} applied to $\lambda\cup\lambda$, since the spherical representations are representations of $\GL_{2n}(\mathbf{F}_q)$.
\end{proof}

\begin{lemma}
\label{Lemma: Dimension Bound}
Fix $\varphi$ with $d(\varphi)=1$ and $i$. Then
\begin{equation*}
\sum _{l(\lambda(\varphi))=n-i}d_{\lambda\cup\lambda}\leq Cq^{2i+4in-2i^2}
\end{equation*}
with $C>0$ independent of $i$. 
\end{lemma}

\begin{proof}
For each $\lambda$ with $l(\lambda(\varphi))=n-i$, if $\widetilde{\lambda}$ denotes the partition valued function equal to $\lambda$ except at $\varphi$, where the first column is removed, then
\begin{equation*}
\sum _{l(\lambda(\varphi))=n-i}d_{\widetilde{\lambda}\cup\widetilde{\lambda}}^2\leq |\GL_{2i}(\mathbf{F}_q)/\Sp_{2i}(\mathbf{F}_q)|<q^{2i^2}.
\end{equation*}
Next, note that there is $C>0$ such that
\begin{equation*}
\frac{q}{q-1}\frac{q^2}{q^2-1}\cdots\frac{q^n}{q^n-1}<C
\end{equation*}
for all $n,q$.

Now
\begin{equation*}
\begin{split}
d_{\lambda\cup\lambda}&\leq \frac{q^{2i}\prod _{j=2i+1}^{2n} (q^j-1)}{\prod _{j=1}^{2n-2i}(q^j-1)}d_{\widetilde{\lambda}\cup\widetilde{\lambda}}
\\&\leq C\frac{q^{2i}\prod _{j=2i+1}^{2n} q^j}{\prod _{j=1}^{2n-2i}q^j}d_{\widetilde{\lambda}\cup\widetilde{\lambda}}
\\&= Cq^{2i}(q^{2i})^{2n-2i}d_{\widetilde{\lambda}\cup\widetilde{\lambda}}
\end{split}
\end{equation*}
and thus
\begin{equation*}
\begin{split}
\sum _{l(\lambda(\varphi))=n-i}d_{\lambda\cup\lambda}&\leq Cq^{4in+2i-4i^2}\sum _{l(\lambda(\varphi))=n-i}d_{\widetilde{\lambda}\cup\widetilde{\lambda}}
\\&\leq Cq^{2i+4in-2i^2}.
\end{split}
\end{equation*}
\end{proof}

Now the terms with small length can be bounded.
\begin{lemma}
\label{Lemma:Small Case}
Let $F$ denote the set of partition valued functions $\lambda$ such that $l(\lambda(\varphi))\leq n-n^{0.6}$ for all $d(\varphi)=1$, and such that $\phi_{\lambda}>0$. Then
\begin{equation*}
\sum _{\lambda\in F}d_{\lambda\cup\lambda}|\phi_{\lambda}|^{2(n+c)}<Ae^{-Bc}
\end{equation*}
for large enough $n$, with constants $A,B>0$.
\end{lemma}

\begin{proof}
First, note that as a consequence of \eqref{eq:1}, the simpler inequality
\begin{equation*}
\phi_{\lambda}\leq \frac{q^{2n-1}}{(q^{2n}-1)(q^{2n-2}-1)}\sum _{d(\varphi)=1}\sum _{(i,j)\in C(\lambda(\varphi))}\frac{(q^{2i}-1)}{(q-1)}
\end{equation*}
holds. Then note that $i\leq l(\lambda(\varphi))$ and so for some $\varepsilon>0$ this is bounded by
\begin{equation*}
\sum _{d(\varphi)=1}n\frac{q^{2l(\lambda(\varphi))}(1+\varepsilon)}{q^{2n-2}}\leq qn\frac{q^{2\max_{d(\varphi)=1} l(\lambda(\varphi))}(1+\varepsilon)}{q^{2n-2}},
\end{equation*}
where $\varepsilon<1$ for large enough $n$.

Now sum over $F$ to obtain
\begin{equation*}
\sum _{\lambda\in F}d_{\lambda\cup\lambda}(\phi_{\lambda})^{2n}\leq \sum _{d(\varphi)=1}\sum _{i=n^{0.6}}^n\sum _{l(\lambda(\varphi))=n-i}d_{\lambda\cup\lambda}\left(\frac{qn(1+\varepsilon)}{q^{2i-2}}\right)^{2n}
\end{equation*}
and apply Lemma \ref{Lemma: Dimension Bound} to obtain the bound
\begin{equation*}
\begin{split}
\sum _{\lambda\in F}d_{\lambda\cup\lambda}(\phi_{\lambda})^{2n}&\leq \sum _{d(\alpha)=1}\sum _{i=n^{0.6}}^n\frac{Cq^{2i+4in-2i^2}(n(1+\varepsilon))^{2n}}{(q^{2i-3})^{2n}}
\\&\leq C(q-1)nq^{n(-2n^{0.2}+8+C'\log n)},
\end{split}
\end{equation*}
where $C'$ is a positive constant. For large enough $n$ this is exponentially small.
\end{proof}

Finally, bound the remaining case of large length. 

\begin{lemma}
\label{lemma:large length}
If $F$ denotes the set of partition-valued functions $\lambda$ such that $\phi_{\lambda}>0$ and there is some $\varphi$ with $d(\varphi)=1$ and $l(\lambda(\varphi))>n-n^{0.6}$, then
\begin{equation*}
\sum _{\lambda\in F}d_{\lambda\cup\lambda}|\phi_{\lambda}|^{2(n+c)}\leq Ae^{-Bc}
\end{equation*}
for constants $A,B>0$ and sufficiently large $n$.
\end{lemma}
\begin{proof}
Consider Equation \eqref{eq:1}, first bounding the term $(n-i,1)$ in $C(\lambda(\varphi))$ corresponding to the $\varphi$ witnessing $\lambda(\varphi)>n^{0.6}$, so $i\leq n^{0.6}$. This gives
\begin{equation*}
\frac{q^{2n-2}(q^{2n-2i}-1)}{(q^{2n}-1)(q^{2n-2}-1)}\leq \frac{1+q^{-n}}{q^{2i}}
\end{equation*}
for large enough $n$.

Now if $D$ denotes all other terms in the sum over $\varphi$ and $C(\lambda(\varphi))$ except for this $(n-i,1)$ term, then
\begin{equation*}
\sum _{D}\frac{(q^j-1)(q^{2i}-1)}{q^{j-1}(q-1)(q^2-1)}\leq q^{2n^{0.6}}
\end{equation*}
because the $q^{2i}$ in the remaining summands have $\sum i\leq n^{0.6}$ and so the sum is at most $q^{2n^{0.6}}$ since $q^x+q^y\leq q^{x+y}$ for $x,y\geq 1$. Then for large enough $n$,
\begin{equation*}
\frac{q^{2n-2}(q^2-1)}{(q^{2n}-1)(q^{2n-2}-1)}q^{2n^{0.6}}\leq \frac{q^{-n}}{q^{2i}}
\end{equation*}
and so
\begin{equation*}
\phi_{\lambda}\leq \frac{1+2q^{-n}}{q^{2i}}.
\end{equation*}

Now, use Lemma \ref{Lemma: Dimension Bound} to conclude
\begin{equation*}
\begin{split}
&\sum _{d(\varphi)=1}\sum _{i=1}^{n^{0.6}}\sum _{l(\lambda(\varphi))=n-i}d_{\lambda\cup\lambda}\phi_{\lambda}^{2(n+c)}
\\\leq & \sum _{d(\varphi)=1}\sum _{i=1}^{n^{0.6}} Cq^{2i+4in-2i^2}(1+2q^{-n})^{2(n+c)}q^{-4i(n+c)}
\\=&(q-1)\sum _{i=1}^{n^{0.6}} C(1+2q^{-n})^{2(n+c)}q^{2i-2i^2-4ic}
\\\leq &\sum _{i=1}^{n^{0.6}} C(1+2q^{-n})^{2(n+c)}q^{3-2i-4ic}.
\end{split}
\end{equation*}
Now note that $(1+2q^{-n})^{2n}\rightarrow 1$ as $n\rightarrow \infty$ and so for large enough $n$
\begin{equation*}
\begin{split}
\sum _{d(\varphi)=1}\sum _{i=1}^{n^{0.6}}\sum _{l(\lambda(\varphi))=n-i}d_{\lambda\cup\lambda}\phi_{\lambda}^{2(n+c)}&\leq C\sum _{i=1}^{n^{0.6}}(1+2q^{-n})^{2c}q^{3-2i-4ic}
\\&\leq Cq^3\sum _{i=1}^{n^{0.6}}(q^{-2-3c})^i
\\&\leq Cq^{1-3c}
\end{split}
\end{equation*}
and this establishes the lemma.
\end{proof}

Finally, Theorem \ref{thm:upper bound} follows easily from Lemmas \ref{lemma: negative case}, \ref{Lemma:Small Case} and \ref{lemma:large length}, noting that these lemmas cover all non-zero terms in Equation \eqref{eq: DS bound}.

\section{Lower Bounds}
\label{sec: lower bound}
This section is devoted to proving Theorem \ref{thm:lower bound}. This is done by showing that if only $n-c$ steps are taken, then the proportion of $\GL_{2n}(\mathbf{F}_q)$ which can be reached is exponentially small, and so the walk cannot be mixed.

Similar ideas appear in \cite{H92}, where it is shown that for the random walk on $\GL_n(\mathbf{F}_q)$, after taking $n-c$ steps the random element still has a large fixed subspace. This relies on results in an unpublished manuscript of Rudvalis and Shinoda about the proportion of such elements in $\GL_n(\mathbf{F}_q)$. It will be shown that if only $n-c$ steps are taken, then $g^TJg-J$ has an isotropic subspace of dimension $n+c$ (that is, a subspace where the form restricts to $0$). This is done by using results in \cite{BKS90} to reduce the problem to the computation in the $\GL_n(\mathbf{F}_q)$ case.

\begin{proposition}
\label{prop:supp estimate}
Let 
\begin{equation*}
A_c=\{g\in \GL_{2n}(\mathbf{F}_q)|g=kg_0,\dim(\ker(g_0-I))\geq n+c, k\in \Sp_{2n}(\mathbf{F}_q)\}.
\end{equation*}
Then
\begin{equation*}
\frac{|A_c|}{|\GL_{2n}(\mathbf{F}_q)|}\leq  \frac{C}{q^c}
\end{equation*}
for some constant $C$ independent of $n$.
\end{proposition}

\begin{proof}
Note that if $g\in A_c$, then $g^TJg-J=g_0^TJg_0-J$ is an alternating form with an isotropic subspace of dimension $n+c$, because $g_0$ fixes a subspace of dimension $n+c$. Now, if $I(\omega)$ denotes a maximal isotropic subspace of the form $\omega$, then
\begin{equation*}
\begin{split}
|A_c|&\leq \sum _{\dim(I(g_\mu^TJg_\mu-J))\geq n+c}|Kg_\mu K|
\\&=\sum _{\dim(\ker(g_\mu-I))\geq n+c}|Kg_\mu K|
\end{split}
\end{equation*}
because the set can be broken up into double cosets as it is $\Sp_{2n}(\mathbf{F}_q)$-bi-invariant since
\begin{equation*}
\dim(I(g_\mu^TJg_\mu-J))=\dim(I((k_1g_\mu k_2)^TJk_1g_\mu k_2-J)).
\end{equation*}
Thus the dimension of the maximal isotropic subspace of
\begin{equation*}
\begin{split}
g_\mu^TJg_\mu-J&=\left(\begin{array}{cc}
M_\mu^T&0\\
0&I\\
\end{array}\right)
\left(\begin{array}{cc}
0&I\\
-I&0\\
\end{array}\right)
\left(\begin{array}{cc}
M_\mu&0\\
0&I\\
\end{array}\right)-\left(\begin{array}{cc}
0&I\\
-I&0\\
\end{array}\right)
\\&=\left(\begin{array}{cc}
0&M_\mu^T-I\\
I-M_\mu&0\\
\end{array}\right)
\end{split}
\end{equation*}
is just
\begin{equation*}
2\dim(\ker(M_\mu-I))+\frac{1}{2}(2n-2\dim(\ker(M_\mu-I)))=n+\dim(\ker(M_\mu-I))
\end{equation*}
and so $\dim(I(g_\mu^TJg_\mu-J))\geq n+c$ is equivalent to $\dim(\ker(M_\mu-I))\geq c$. 

But $|Kg_\mu K|=|\Sp_{2n}(\mathbf{F}_q)||C_\mu|_{q\mapsto q^2}$ by Proposition \ref{prop: double coset size} and so
\begin{equation*}
|A_c|=|\Sp_{2n}(\mathbf{F}_q)| \sum _{\dim(\ker(M_\mu-I))\geq c}|C_{\mu}|_{q\mapsto q^2}.
\end{equation*}
Next, write
\begin{equation*}
|C_{\mu}|_{q\mapsto q^2}= |C_{\mu}|\frac{|C_{\mu}|_{q\mapsto q^2}}{|C_{\mu}|}
\end{equation*}
and note that
\begin{equation*}
\begin{split}
\frac{|C_{\mu}|_{q\mapsto q^2}}{|C_{\mu}|}&=\frac{|\GL_n(\mathbf{F}_q)|_{q\mapsto q^2}}{|\GL_n(\mathbf{F}_q)|}\frac{a_\mu(q)}{a_\mu(q^2)}
\\&= \frac{|\GL_n(\mathbf{F}_q)|_{q\mapsto q^2}}{|\GL_n(\mathbf{F}_q)|}\frac{q^{n+2n(\mu)}\prod \psi_{m_i(\mu)}(q^{-1})}{q^{2n+4n(\mu)}\prod \psi_{m_i(\mu)}(q^{-2})}
\\&\leq \frac{|\GL_n(\mathbf{F}_q)|_{q\mapsto q^2}}{|\GL_n(\mathbf{F}_q)|}\frac{1}{q^n\prod_{j} \prod_{i=1}^{m_j(\mu)} (1+q^{-i})}
\\&\leq \frac{|\GL_n(\mathbf{F}_q)|_{q\mapsto q^2}}{|\GL_n(\mathbf{F}_q)|}q^{-n}.
\end{split}
\end{equation*}
Then
\begin{equation*}
\begin{split}
&|\Sp_{2n}(\mathbf{F}_q)| \sum _{\dim(\ker(M_\mu-I))\geq c}|C_{\mu}|_{q\mapsto q^2}
\\\leq & |\Sp_{2n}(\mathbf{F}_q)|\frac{|\GL_n(\mathbf{F}_q)|_{q\mapsto q^2}}{|\GL_n(\mathbf{F}_q)|}q^{-n}\sum _{\dim(\ker(M_\mu-I))\geq c}|C_{\mu}|.
\end{split}
\end{equation*}
Now from \cite[\S 6]{H92}
\begin{equation*}
\sum _{\dim(\ker(M_\mu-I))\geq c}|C_{\mu}|\leq \frac{4}{q^c}|\GL_n(\mathbf{F}_q)|
\end{equation*}
and note that
\begin{equation*}
\begin{split}
\frac{|\Sp_{2n}(\mathbf{F}_q)||\GL_n(\mathbf{F}_q)|_{q\mapsto q^2}q^{-n}}{|\GL_{2n}(\mathbf{F}_q)|}&= \frac{q^{n^2}\prod _{i=1}^n(q^{2i}-1)q^{n^2-n}\prod _{i=1}^n(q^{2i}-1)}{q^nq^{n(2n-1)}\prod _{i=1}^{2n}(q^i-1)}
\\&=\prod _{i=1}^n\frac{q^{2i}-1}{q^{2i}-q}
\end{split}
\end{equation*}
and this is bounded (independent of $n$ and $q$) by some $C$. Thus,
\begin{equation*}
\frac{|A_c|}{|\GL_{2n}(\mathbf{F}_q)|}\leq \frac{4C}{q^c}.
\end{equation*}
\end{proof}

Then Theorem \ref{thm:lower bound} follows easily because $P^{\ast(n-c)}$ is supported on $A_c$.

\begin{proof}[Proof of Theorem \ref{thm:lower bound}]
First, note that $P^{\ast(n-c)}$ is supported on $A_c$ because the random element may be written as
\begin{equation*}
k_0g_\mu k_1\cdots k_{n-1}g_\mu k_n =k_0'\prod_{i=1}^{n-c}\left((k_i')^{-1} g_\mu k_i'\right),
\end{equation*}
where $k_n'=k_n$, and $k_i'=k_ik_{i+1}'$, and note that each $(k_i')^{-1} g_\mu k_i'$ is a transvection, and so $\prod_{i=1}^{n-c} (k_i')^{-1} g_\mu k_i'$ is the product of $n-c$ transvections, and thus has an $(n+c)$-dimensional $1$-eigenspace. Thus, $P^{\ast(n-c)}$ is supported on $A_c$. Then as the support of $D$ has only $q$ elements,
\begin{equation*}
\begin{split}
\|P^{\ast(n-c)}\ast D-U\|&\geq |P^{\ast(n-c)}\ast D(A_c\cdot \mathrm{supp}(D))-U(A_c\cdot \mathrm{supp}(D))|
\\&\geq 1-\frac{q|A_c|}{|\GL_{2n}(\mathbf{F}_q)|}
\\&\geq 1-\frac{C}{q^{c-1}}
\end{split}
\end{equation*}
by Proposition \ref{prop:supp estimate}.
\end{proof}

\bibliography{bibliography}{}
\bibliographystyle{amsplain}

\end{document}